\documentclass[11pt,centertags,reqno,twoside]{amsart}
\usepackage{amssymb, graphicx}
\usepackage{mathptmx}
\usepackage[mathcal]{euscript}

\usepackage[english]{babel}

\DeclareSymbolFont{SY}{U}{psy}{m}{n}
\DeclareMathSymbol{\emptyset}{\mathord}{SY}{'306}

\allowdisplaybreaks

\numberwithin{equation}{section}

\newtheorem{theorem}{Theorem}

\newtheorem{lemma}[theorem]{Lemma}

\theoremstyle{definition}
\newtheorem{definition}[theorem]{Definition}
\theoremstyle{remark}
\newtheorem{remark}[theorem]{Remark}

\begin{document}

\title[Spectral asymptotics for two--term even--order differential operators with
homogeneous delay]
{Spectral asymptotics for two--term even--order differential operators with
homogeneous delay}

\author[D. M. Polyakov]
{Dmitry M. Polyakov}

\address{Dmitry M. Polyakov\newline\hspace*{9mm} Southern Mathematical Institute, Vladikavkaz Scientific Center of
RAS\newline\hspace*{9mm}
Vladikavkaz, 53 Vatutin str., Russia\newline\hspace*{9mm}
Bashkir State Pedagogical University named after M. Akhmulla\newline\hspace*{9mm}
Ufa, 3a Oktyabrskoj rev. str., Russia}
\email{DmitryPolyakow@mail.ru}

\keywords{two--term even--order differential operator, homogeneous delay,
eigenvalue asymptotics, non--local operator}.

\begin{abstract}
We consider two--term even--order operators on the unit interval with
homogeneous delay and with Dirichlet and Neumann boundary conditions. The main result
provides to the eigenvalue asymptotics of this operator with respect to all indices
enumerating the eigenvalues. This asymptotic formula is uniform in the parameter
of the homogeneous delay. We also discuss the nontrivial
high--frequency phenomenon demonstrated by the uniform spectral asymptotics.
\end{abstract}

\maketitle

\section{Introduction}\label{sec1}

Functional--differential equations are important and interesting objects of study.
Problems containing contractions (or homogeneous delay) and dilatations of
independent variables arise in the theory of nonlinear oscillations, astrophysics,
modeling of cell growth processes, number theory, and probability theory,
see~\cite{Ambartsumyan}, \cite{Ockendon}, \cite{Hall}, \cite{Gaver}.
The general theory of boundary value
problems for elliptic functional--differential equations with variable contractions
was constructed in \cite{Ros2011}, \cite{Ros2017}, \cite{RosTovs2019}, \cite{RosTovs2022}.
These papers also concern the results about the existence, uniqueness, and
smoothness of solutions to this class of equations.

While the theory of functional--differential equations with contractions is quite well--developed, the spectral
properties of corresponding operators is almost unexplored. In particular, in~\cite{PodSkub} and~\cite{Skub} for
strongly elliptic equations with contractions, the completeness and basis property for the system of
eigenfunctions and associated functions were established. We also mention the manuscript \cite{Ros2017},
where the author considered the elliptic operators with contractions in domains with small boundary perturbations.
In that article it was established that the eigenvalues of these operators converge to the eigenvalues of the operators
in the limiting domains, estimates of convergence rates were determined, and the stability of
the eigenvalues of the Neumann problem for the corresponding equations was proved.

Some aspects of inverse spectral problems for Sturm--Liouville type
operators with homogeneous delay were considered in \cite{Nedic},
\cite{Pikula2014}, \cite{Pikula2013}, \cite{Djuric}. In particular,
the leading term in the eigenvalue asymptotics for this operator was obtained
in \cite{Djuric}, and some next--to--leading terms were established in \cite{Nedic},
\cite{Pikula2014}, \cite{Pikula2013} for fixed values of the homogeneous delay parameter.
However, the estimates of the error terms are not uniform in this parameter.

In our article we study operators with homogeneous delay in the lower order
term. Namely, we consider the even--order operator on an interval with
Dirichlet type and Neumann type boundary conditions perturbed by an operator
with homogeneous delay. The contraction can be arbitrary, including the case
of an arbitrary small contraction. Note that this operator is more general than
the operator in \cite{Nedic}, \cite{Pikula2014}, \cite{Pikula2013}, \cite{Djuric}.

Our main result is the eigenvalue asymptotics, which is uniform in
the homogeneous delay parameter. The uniformity of the error term is
important for the following reason. The presence of a parameter in the perturbation
entails a natural dependence of the error term of the spectral asymptotics
on the index and on this parameter.
However, the estimate for this term can be destroyed if the index and parameter
grow simultaneously. As a simplest example, consider the function $(n\alpha)^{-2}$.
For each fixed $\alpha$ and large $n$, this function tends to zero as
$\mathcal{O}(n^{-2})$. But if $\alpha$ is also reduced (for example, $\alpha=n^{-1}$)
this function has order $\mathcal{O}(1)$, while for $\alpha=n^{-2}$
it grows. Therefore, the most objective approach is to determine uniform
spectral asymptotics in the parameter $\alpha$. Note that none of the papers
\cite{Nedic}, \cite{Pikula2014}, \cite{Pikula2013}, \cite{Djuric} on the
Sturm--Liouville operator with homogeneous delay address this issue.

We emphasize that classical methods of perturbation theory do not in their standard form provide uniform spectral asymptotics.
In order to address this problem we develop a method based on the method of similar operators (see~\cite{BaskKrUsk}).
That method provides general conditions for the similarity of the considered operator to an operator with a simpler
structure (see~Section~\ref{sec4}). However, even for classical differential operators verifying these conditions is nontrivial
and requires new approaches, which are typically cumbersome and involve additional constructions. In this article we propose a new
simple approach that differs significantly from previous ones (cf.~\cite{BorPolIzv}, \cite{PolyakovM2AS}). It allows
direct use of the abstract results of \cite{BaskKrUsk}. Moreover, almost all
methods of perturbation theory (and, in particular, the method of similar operators)
are designed to handle spectral asymptotics only for large indices. Our approach, however,
describes the ensemble of eigenvalues for \emph{all indices}. Thus, since in the
particular case our nonlocal operator is an ordinary differential operator, we
present a fairly effective method for obtaining asymptotic formulas for all eigenvalues.

\section{Statement of problem and the main results}

We consider a self--adjoint operator $\mathcal{A}: \mathfrak{D}(\mathcal{A})\subset L^2(0, 1)
\to L^2(0, 1)$ of the form $\mathcal{A}y=(-1)^ky^{(2k)}$ on the domain
$\mathfrak{D}(\mathcal{A})=\{y\in W_2^{2k}(0, 1)\}$ with the following boundary conditions
\begin{align*}
&\text{Dirichlet type:}\quad y^{(2j)}(0)=y^{(2j)}(1)=0, \quad 0\leqslant
j\leqslant k-1, \\
&\text{Neumann type:}\quad y^{(2j+1)}(0)=y^{(2j+1)}(1)=0,
\quad 0\leqslant j\leqslant k-1.
\end{align*}
Now we define the contraction operator in $L^2(0, 1)$ as
\[
(\mathcal{L}(\alpha)y)(x)=y(\alpha x), \qquad x\in (0, 1),
\]
where $\alpha\in (0, 1)$ is a parameter.

Let $V$ be a complex--valued function in the space $L^\infty(0, 1)$ with a sufficiently small norm.
We introduce an operator $\mathcal{B}(\alpha)$ acting into $L^2(0, 1)$ by the rule
\begin{equation}\label{B}
\mathcal{B}(\alpha)=-V\mathcal{L}(\alpha),
\end{equation}
where $\alpha\in (0, 1)$. Note that if $\alpha=1$, then $\mathcal{B}(1)$ becomes the
classical operator of multiplication by $V$. This case is studied in
detail in \cite{PolyakovM2AS}. If $\alpha=0$, then operator $\mathcal{H}(0)$
is the operator with frozen argument. This degenerate case is of a little interest
here, as $\mathcal{B}(0)$ is a finite--rank perturbation.
The operators with frozen argument are a separate class
with interesting properties (see articles \cite{BondButVas}, \cite{BondHuYang}
for the second--order operators). This requires independent consideration
for higher--order operators and we do not consider such problem in this paper.

In this article we study the behavior of spectrum of the nonlocal
operator $\mathcal{H}(\alpha)=\mathcal{A}-\mathcal{B}(\alpha)$ of the form
\[
(\mathcal{H}(\alpha)y)(x)=(-1)^ky^{(2k)}(x)+V(x)y(\alpha x)
\]
for $y\in\mathfrak{D}(\mathcal{H}(\alpha))$ on the domain
$\mathfrak{D}(\mathcal{H}(\alpha)):=\mathfrak{D}(\mathcal{A})$
with the Dirichlet or Neumann types of boundary conditions.
Our preliminary result describes the basic properties of the operator
$\mathcal{H}(\alpha)$.
\begin{theorem}\label{th0}
Suppose that $V\in L^\infty(0, 1)$ and
\begin{equation}\label{Vc1lalpha}
\|V\|_{L^\infty(0, 1)}\leqslant c_0\sqrt{\alpha},
\end{equation}
where $\alpha\in (0, 1)$ and $c_0>0$ is some constant independent of $\alpha$.
The operator $\mathcal{H}(\alpha)$ is $m$--sectorial
and the associated sectorial closed form in the space $L^2(0, 1)$ is defined by the identity
\begin{equation}\label{2.1}
\mathfrak{h}^\alpha(u, v)=(u^{(k)}, v^{(k)})_{L^2(0, 1)}
+(V\mathcal{L}(\alpha)u, v)_{L^2(0, 1)}
\end{equation}
on the domain $\mathfrak{D}(\mathfrak{h}^\alpha)=\{u\in W_2^k(0, 1)\}$.
There exists a number $\Lambda$ independent of $\alpha$ such that the half--plane
$\mathrm{Re}\,\lambda\leqslant \Lambda$ is in the resolvent set of the operator
$\mathcal{H}(\alpha)$ for each $\alpha\in (0, 1)$. Moreover, the operator
$\mathcal{H}(\alpha)$ has a compact resolvent and its spectrum consists of
countably many eigenvalues with the only accumulation point at infinity.
\end{theorem}

Theorem~\ref{th0} shows that the eigenvalues of the operator $\mathcal{H}(\alpha)$
are located in the half--plane $\mathrm{Re}\,\lambda\geqslant \Lambda$.
We arrange them in the ascending order of their absolute values and denote
by $\lambda_n$, $n\in\mathbb{N}$.

Now we formulate the main result of our manuscript. It is devoted to the eigenvalue
asymptotics for the operator $\mathcal{H}(\alpha)$.
\begin{theorem}\label{th1}
Assume that $V\in L^\infty(0, 1)$ and satisfies the estimate \eqref{Vc1lalpha}
with $\alpha\in (0, 1)$ and $c_0<1/8$. Then
the eigenvalues $\lambda_n$ of the operator $\mathcal{H}(\alpha)$ for the Dirichlet type boundary conditions
have the following asymptotics
\begin{equation}\label{asymptDir}
\lambda_n=(\pi n)^{2k}+2\int_0^1V(x)\sin\pi n\alpha x\sin \pi nx\,dx
-\frac{1}{\pi^{2k}}\sum\limits_{\substack{j=1 \\ j\ne n}}^\infty
\frac{b_{nj}^Db_{jn}^D}{j^{2k}-n^{2k}}+\mathcal{O}(n^{-4k+2}),
\end{equation}
where
\begin{equation}\label{bnjD}
b_{pj}^D=-2\int_0^1V(x)\sin\pi j\alpha x\sin \pi px\,dx, \qquad p, j\geqslant 1.
\end{equation}
In the case of the Neumann type boundary conditions the asymptotics reads
\begin{equation}\label{asymptNeu}
\lambda_n=(\pi n)^{2k}+2\int_0^1V(x)\cos\pi n\alpha x\cos\pi nx\,dx
-\frac{1}{\pi^{2k}}\sum\limits_{\substack{j=0 \\ j\ne n}}^\infty
\frac{b_{nj}^Nb_{jn}^N}{j^{2k}-n^{2k}}+\mathcal{O}(n^{-4k+2}),
\end{equation}
where
\begin{equation}\label{bnjN}
b_{pj}^N=-2\int_0^1V(x)\cos\pi j\alpha x\cos\pi px\,dx, \qquad p, j\geqslant 0.
\end{equation}
The estimate for the error terms is uniform in $\alpha\in (0, 1)$ and $n$.
\end{theorem}
\begin{remark}
Theorem~\ref{th1} also holds for $\alpha=1$. However, in this case, all
restrictions on $V$ are redundant. The asymptotics can be determined
for an arbitrary function $V\in L^2(0, 1)$. This result for the case of Dirichlet type
boundary conditions was obtained in \cite[Theorem~3]{PolyakovM2AS}.
The methodology from that manuscript can be repeated completely for the case
of Neumann type boundary condition.
\end{remark}
\begin{remark}\label{rem2}
The requirement that the constant $c_0$ be small can be relaxed. It can be any number
independent of $\alpha$. However, in this case the formulas \eqref{asymptDir} and
\eqref{asymptNeu} describe the behavior of the entire ensemble of eigenvalues
uniformly in $\alpha$, starting from some sufficiently large index.
The main distinctions and the necessary proofs are given in Appendix.
\end{remark}

Now we briefly discuss the main result. In Theorem~\ref{th1} we obtain a uniform
spectral asymptotics \eqref{asymptDir} for Dirichlet type boundary conditions
and \eqref{asymptNeu} for Neumann type boundary conditions. We emphasize once again that
this asymptotic formulas hold for \textit{all indices} $n$, not only for sufficiently large ones.
At the same time the error terms in \eqref{asymptDir} and \eqref{asymptNeu}
are bounded by the quantity $\widetilde{C}_1n^{-4k+2}$ for all $n$ with a constant
$\widetilde{C}_1$ independent of $\alpha$ and $n$. Therefore, we obtain the behavior of the
\textit{entire} ensemble of eigenvalues with the indices $n$ uniformly in $\alpha$.
The next main feature is the presence of a second approximation
term (namely, the terms with a sum). This term is bounded by the quantity
$\widetilde{C}_2n^{-2k+1}$, where the constant $\widetilde{C}_2$ is also
independent of $\alpha$ and $n$. Thus, we observe the decay rate of
the second approximation and the error term depending on the order of the differential
operator. Note that this is not always possible even for ordinary differential
operators (see~\cite{Akhmerova}).

If $\alpha=1$, then the asymptotics \eqref{asymptDir} becomes the known asymptotics
for a two--term even--order operator on an interval with the Dirichlet type
boundary conditions (see~\cite{PolyakovM2AS}).

In the case $\alpha\in (0, 1)$ the asymptotics~\eqref{asymptDir} and \eqref{asymptNeu}
exhibit qualitatively different behavior. In these formulas the high--frequency
phenomenon is not isolated in a separate term, but is structurally contained in
the second term. This significantly distinguishes it from the case where
the perturbation involves a translation (cf.~\cite{BorPolIzv}, \cite{BorPolM2AS}).
The second feature is the term with the sum in~\eqref{asymptDir} and \eqref{asymptNeu}.
Unlike in \cite{BorPolIzv} and \cite{PolyakovM2AS}, this term cannot be expressed
in an integral form that explicitly captures the interaction between the index $n$ and the
parameter~$\alpha$. This is essentially due to the nature of the perturbation
with homogeneous delay and it constitutes a second significant difference from the
case of a translation perturbation (cf. again \cite{BorPolIzv}, \cite{BorPolM2AS}).
Note that numerical examples show that this term
is fairly small and does not significantly affect the eigenvalue asymptotic
(in the examples given below, it is of the order of $10^{-6}$), and
the main contribution comes from the integral term in~\eqref{asymptDir}
and \eqref{asymptNeu}.

Now we consider the behavior of the eigenvalues for a specific example
as different values of $\alpha$. Let $k=2$, the Dirichlet type boundary conditions,
and
\[
V(x)=\frac{e^{ix}\sqrt{\alpha}}{20}.
\]
Then in this case $c_0=1/20<1/8$ and all the conditions of Theorem~\ref{th1} are satisfied.
The following two figures are given by $\alpha=0.1$ and $\alpha=0.9$.
The eigenvalues $\lambda_n$ are defined by \eqref{asymptDir} of the operator
$\mathcal{H}(\alpha)$ in all cases are localized along the red curve in the
complex plane. In addition to the curve, we also show three series of points.
The first series of points indicated by the blue squares corresponds the
asymptotics \eqref{asymptDir} in the case of $\alpha=0$. As we note above this situation is a degenerate case and
the eigenvalues of $\mathcal{H}(\alpha)$ coincide with the eigenvalues of the operator $\mathcal{A}$.
The second series indicated by green crosses depict the values of sums of the first three
terms in well--known asymptotics \eqref{asymptDir} if $\alpha=1$. Finally,
the series indicated by the red squares corresponds the asymptotics \eqref{asymptDir}.
\begin{figure}[t]
\begin{center}
\includegraphics[scale=0.5]{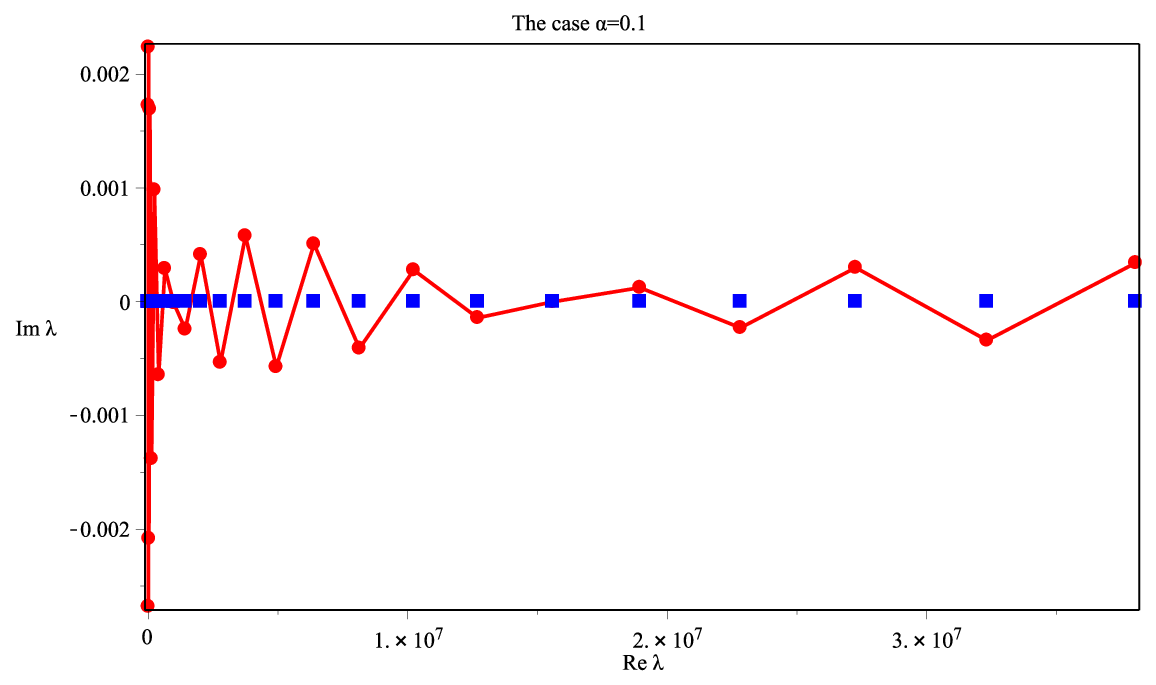}
\caption{\small The red points represent the leading terms of \eqref{asymptDir}
for the operator $\mathcal{H}(\alpha)$ with $V(x)=e^{ix}\sqrt{\alpha}/20$ and $k=2$. These points are connected by a red
curve to show their dependence on the index $n$. The blue squares correspond to the spectrum $(\pi n)^4$
of the operator $\mathcal{A}$. This Figure corresponds to the case $\alpha=0.1$.}
\end{center}
\end{figure}
\begin{figure}[t]
\begin{center}
\includegraphics[scale=0.5]{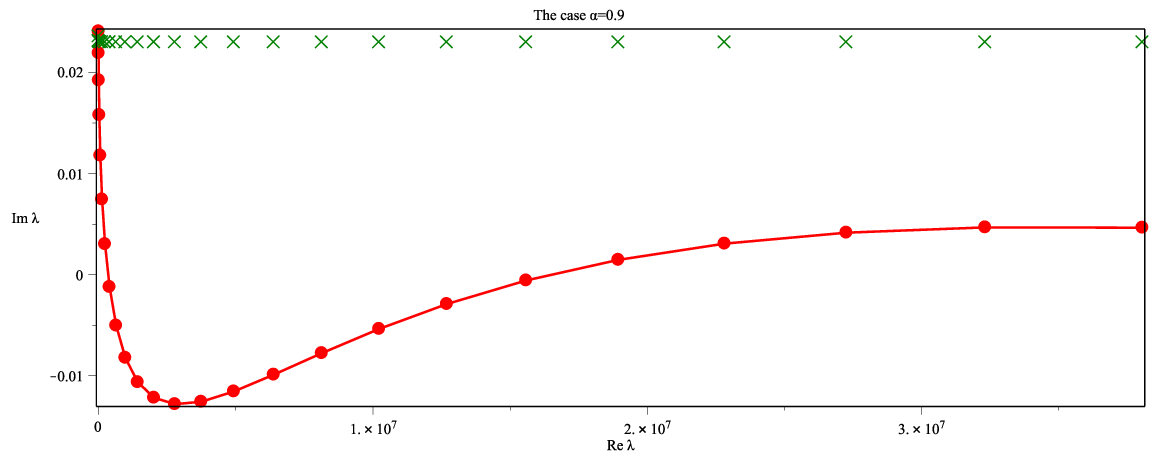}
\caption{\small This Figure corresponds to the case $\alpha=0.9$. The red points represent the terms of
\eqref{asymptDir} for the operator $\mathcal{H}(\alpha)$ with $V(x)=e^{ix}\sqrt{\alpha}/20$ and $k=2$.
These points are again connected by a red curve. The green crosses correspond to the well--known asymptotics
\eqref{asymptDir} for $\alpha=1$.}
\end{center}
\end{figure}

The figures reveal a significant gap between the terms in \eqref{asymptDir}
and the stationary cases for $\alpha=0$ and $\alpha=1$. This gap manifests
the high--frequency phenomenon described above. However, this effect behaves
differently for different $\alpha$. In particular, for $\alpha=0.1$ Figure 1 clearly
shows oscillations that decay as the eigenvalue index increases and tend to the
eigenvalues of the unperturbed operator $\mathcal{A}$.
This is because $\sin\pi n\alpha x$ varies slowly, so the integral term makes a
significant contribution to the asymptotic expression. Moreover, the points in
\eqref{asymptDir} are closest to the eigenvalues of $\mathcal{A}$ precisely when
$n\alpha$ is an integer.

A completely opposite phenomenon occurs at $\alpha=0.9$. Here $\sin\pi n\alpha x$
is close to $\sin\pi nx$ and is a rapidly oscillating function. The integral averages
out and is approximately equal to $\int_0^1V(x)\,dx$, as can be seen in Figure 2.
The imaginary part of $\lambda_n$ is relatively small and decays monotonically, approaching
the value $\mathrm{Im}\,\int_0^1V(x)\,dx$ as $n$ increases. Consequently,
oscillations are virtually absent in this case.

Based on our observations, we conclude that the high--frequency phenomenon in the case
of the operator with homogeneous delay consists of damped oscillations of the
eigenvalues, with amplitude decaying as $\mathcal{O}(n^{-1})$ for a fixed
$\alpha\in (0, 1)$. This qualitatively differs from the case where a translation
operator is present instead of homogeneous delay. In particular, \cite{BorPolIzv} and
\cite{BorPolM2AS} studied the Schr\"odinger operator with a translation in
the lowest--order term. The eigenvalue asymptotics in that case contained
a nontrivial high--frequency term of order $\mathcal{O}(1)$. Hence, homogeneous delay
produces a less pronounced high--frequency phenomenon. Moreover, unlike the
translation operator, the homogeneous delay operator does not converge in the
uniform operator norm as $\alpha\to 1$, therefore, it is not continuous in the uniform
operator topology.

We do not treat the case of Neumann type boundary conditions separately,
as it is handled in a similar manner and the corresponding operator exhibits the
same nontrivial effects. Their nature is exactly the same as in the Dirichlet case.
Essentially, the resulting graphics will be shifted, but the overall structure
will be preserved.

\section{$m$--sectoriality and compactness of resolvent}\label{sec3*}

In this section we prove Theorem~\ref{th0}. Now we consider the
sesquilinear form~\eqref{2.1} in $L^2(0, 1)$ on the domain
$\mathfrak{D}(\mathfrak{h}^\alpha)=\{u\in W_2^k(0, 1)\}$. Using the properties
of the operator $\mathcal{L}$, we obtain that this form can be rewritten as
\[
\mathfrak{h}^\alpha(u, v)=\int_0^1 u^{(k)}(x)\overline{v^{(k)}(x)}\,dx
+\int_0^1 V(x)u(\alpha x)\overline{v(x)}\,dx.
\]
Now we estimate $\mathrm{Re}\,\mathfrak{h}^\alpha(u, u)$ and
$\mathrm{Im}\,\mathfrak{h}^\alpha(u, u)$. Using the H\"older inequality and
\eqref{Vc1lalpha}, we have
\begin{align*}
\bigg|\int_0^1 V(x)u(\alpha x)\overline{u(x)}\,dx\bigg| &\leqslant
\|V\|_{L^\infty(0, 1)}\bigg(\int_0^1|u(\alpha x)|^2\,dx\bigg)^\frac{1}{2}
\|u\|_{L^2(0, 1)} \\
&\leqslant \frac{\|V\|_{L^\infty(0, 1)}\|u\|_{L^2(0, 1)}^2}{\sqrt{\alpha}}\leqslant
c_0\|u\|_{L^2(0, 1)}^2,
\end{align*}
where $u\in\mathfrak{D}(\mathfrak{h}^\alpha)$ and $c_0>0$ is independent
of $\alpha$. This estimate and standard interpolation inequalities immediately give
\begin{equation}\label{estReIm}
\mathrm{Re}\,\mathfrak{h}^\alpha(u, u)+c_1\|u\|^2_{L^2(0, 1)}\geqslant
\frac{1}{2}\|u\|^2_{W^k_2(0, 1)}, \qquad
|\mathrm{Im}\,\mathfrak{h}^\alpha(u, u)|\leqslant c_0\|u\|_{L^2(0, 1)}^2,
\end{equation}
for any $u\in\mathfrak{D}(\mathfrak{h}^\alpha)$, where $c_0$ and $c_1$
are some positive constants independent of $\alpha$. Thus,
\[
|\mathrm{Im}\,\mathfrak{h}^\alpha(u, u)|\leqslant
c_0\Big(\mathrm{Re}\,\mathfrak{h}^\alpha(u, u)+c_1\|u\|^2_{L^2(0, 1)}\Big),
\qquad u\in\mathfrak{D}(\mathfrak{h}^\alpha),
\]
and, therefore,
\begin{equation}\label{subset}
\mathfrak{h}^\alpha(u, u)\subseteq \{z\in\mathbb{C}:
|\mathrm{Im}\,z|\leqslant c_0(\mathrm{Re}\,z+c_1)\}, \qquad
u\in\mathfrak{D}(\mathfrak{h}^\alpha), \qquad \|u\|_{L^2(0, 1)}=1.
\end{equation}
Hence, the form $\mathfrak{h}^\alpha$ is sectorial. Moreover, the estimates
\eqref{estReIm} show that this form is closed. It follows from the first
representation theorem \cite[Ch. V\!I, Sect. 2.1, Thm.~2.1]{Kato} that there
exists an $m$--sectorial operator is associated with this form. We denote this operator by
$\widetilde{\mathcal{H}}(\alpha)$.

It remains to show that $\widetilde{\mathcal{H}}(\alpha)=\mathcal{H}(\alpha)$.
Let $u\in\mathfrak{D}(\mathcal{H}(\alpha))$, $v\in\mathfrak{D}(\mathfrak{h}^\alpha)$.
Using integration by parts, we obtain
\[
(\mathcal{H}(\alpha)u, v)_{L^2(0, 1)}=\mathfrak{h}^\alpha(u, v).
\]
This gives that $\mathfrak{D}(\mathcal{H}(\alpha))\subseteq
\mathfrak{D}(\widetilde{\mathcal{H}}(\alpha))$ and the operator
$\widetilde{\mathcal{H}}(\alpha)$ is an extension of the operator
$\mathcal{H}(\alpha)$.

Now we prove the reverse inclusion $\mathfrak{D}(\mathcal{H}(\alpha))\supseteq
\mathfrak{D}(\widetilde{\mathcal{H}}(\alpha))$. According to the first representation
theorem \cite[Ch. V\!I, Sect.~2.1, Thm.~2.1]{Kato}, the domain of the operator
$\widetilde{\mathcal{H}}(\alpha)$ consists of the functions
$u\in\mathfrak{D}(\mathfrak{h}^\alpha)$ such that there exists an associated
function $h\in L^2(0, 1)$ satisfying the identity
\begin{equation}\label{3.1}
\mathfrak{h}^\alpha(u, v)=(h, v)_{L^2(0, 1)}\qquad\text{for all}\qquad
v\in\mathfrak{D}(\mathfrak{h}^\alpha).
\end{equation}
Let $u$ be one of such functions. It follows from the definition of the form
in \eqref{2.1} that the identity \eqref{3.1} can be rewritten as
\[
(u^{(k)}, v^{(k)})_{L^2(0, 1)}=(\widetilde{h}, v)_{L^2(0, 1)}\quad\text{for all}
\quad v\in\mathfrak{D}(\mathfrak{h}^\alpha),\qquad
\widetilde{h}:=h-V\mathcal{L}u\in L^2(0, 1).
\]
Therefore, the function $u$ is a generalized solution to the problems
\[
(-1)^ku^{(2k)}=\widetilde{h} \quad \text{on} \quad (0, 1), \qquad u^{(2j)}(0)=u^{(2j)}(1)=0,
\quad 0\leqslant j\leqslant k-1,
\]
or
\[
(-1)^ku^{(2k)}=\widetilde{h} \quad \text{on} \quad (0, 1),
\qquad u^{(2j+1)}(0)=u^{(2j+1)}(1)=0,
\quad 0\leqslant j\leqslant k-1,
\]
and by the standard smoothness improving theorems of solutions
to elliptic boundary value problems we immediately conclude that
$\mathfrak{D}(\mathcal{H}(\alpha))=\{y\in W^{2k}_2(0, 1)\}$ with
boundary conditions of Dirichlet or Neumann types. This gives
$\mathfrak{D}(\widetilde{\mathcal{H}}(\alpha))\subseteq
\mathfrak{D}(\mathcal{H}(\alpha))$. Therefore,
$\mathfrak{D}(\widetilde{\mathcal{H}}(\alpha))
=\mathfrak{D}(\mathcal{H}(\alpha))$ and
$\widetilde{\mathcal{H}}(\alpha)=\mathcal{H}(\alpha)$. Thus, the operator
$\mathcal{H}(\alpha)$ is associated with the sectorial form~$\mathfrak{h}^\alpha$,
and it is hence $m$--sectorial. The relation \eqref{subset} shows that there exists
a chosen $\Lambda\in\mathbb{R}$ independent of $\alpha$ and such that the
half--plane $\{\lambda\in\mathbb{C}: \mathrm{Re}\,\lambda\leqslant\Lambda\}$
is in the resolvent set of the operator $\mathcal{H}(\alpha)$ for each $\alpha\in (0, 1)$.

The resolvent of the operator $\mathcal{H}(\alpha)$ is a bounded operator in
$L^2(0, 1)$. It follows from the Banach theorem on inverse operator that it is a
bounded operator from $L^2(0, 1)$ into $W_2^k(0, 1)$. The compactness of the
embedding of the space  $W_2^k(0, 1)$ into $L^2(0, 1)$ implies the
compactness of  resolvent of operator $\mathcal{H}(\alpha)$.
Using the theorem on operators with compact resolvent
\cite[Ch.~3, Sect.~6.8, Thm.~6.29]{Kato}, the spectrum of the operator
$\mathcal{H}(\alpha)$ consists of countably many eigenvalues
with the only accumulation point at infinity. Theorem~\ref{th0} is proved.

\section{Basic facts about the method of similar operators}\label{sec4}

Before we begin our study of spectral properties of the operator $\mathcal{H}(\alpha)$,
we give the basic principles of the primary research method. As noted in Introduction,
it is the method of similar operators. The main idea of this method is to construct
a similarity transformation between original operator and some block--diagonal
operator with simple structure.

Now we briefly formulate the main notation of the method of similar operators
for some abstract operator such that its spectral properties coincide
with the properties of the operator $\mathcal{H}(\alpha)$. Note that all facts
in this section can be found in \cite[\S~2, 3]{BaskKrUsk}.

Let $H$ be a separable complex Hilbert space and $B(H)$ be the Banach space of
all bounded linear operators in $H$. The symbol $A: D(A)\subset H\to H$
stands for some abstract self--adjoint, closed, densely defined linear operator.
By $\sigma(A)$ and $\rho(A)$ we denote the spectrum and the resolvent set of
the operator $A$. Suppose that
\[
\sigma(A)=\bigcup_{n\geqslant 0}\{\lambda_n\},
\]
where $\lambda_n$ is a simple eigenvalue of $A$.
Moreover, $AP_n=\lambda P_n$, where $P_n=P(\{\lambda_n\}, A)$ is the spectral Riesz
projection corresponding to the spectral component $\{\lambda_n\}$. We shall refer
to the operator $A$ as unperturbed operator. As a perturbation we consider some
linear operator $B: D(B)\subset H\to H$ is subordinated to $A$. The operator
$B$ is subordinate to A if $D(B)\supseteq D(A)$ and $\|B\|_A=\inf\{c_2>0:
\|Bx\|\leqslant c_2\big(\|x\|+\|Ax\|\big), x\in D(A)\}<\infty$. The space of these
operators with domain equal $D(A)$ is the Banach space with respect to the norm
$\|\cdot\|_A$ and we denote it by $\mathfrak{L}_A(H)$.

Now we recall some basic facts about similar operators.
\begin{definition}\label{def1}
Two linear operators $A_j: D(A_j)\subset H\to H$, $j=1, 2$, are called similar,
if there exists a continuously invertible operator $U\in B(H)$ such that
$A_1Ux=UA_2x$, $x\in D(A_2)$, $UD(A_2)=D(A_1)$. The operator $U$ is said
similarity transform of $A_1$ into $A_2$.
\end{definition}

Directly from Definition~\ref{def1}, we have the following result.
\begin{lemma}\label{lemsim}
Suppose that $A_j: D(A_j)\subset H\to H$, $j=1, 2$, are two similar operators
and $U$ is the similarity transform of $A_1$ into $A_2$. Then
$\sigma(A_1)=\sigma(A_2)$, $\sigma_p(A_1)=\sigma_p(A_2)$, $\sigma_c(A_1)=\sigma_c(A_2)$,
where $\sigma_p$ and $\sigma_c$ are the point spectrum and the continuous spectrum,
respectively.
\end{lemma}

The main object of this section is the perturbed operator $A-B: D(A)\subset H\to H$.
The goal of the method is to determine an operator $Q$ such that $A-B$ is similar
to $A-Q$ and the spectral properties of $A-Q$ are in some sense close to those of $A$.
Now we introduce the key notion of the method of similar operators.
\begin{definition}[\cite{BaskKrUsk}]\label{defadm}
Suppose that $A$ is a closed, densely defined linear operator, $\mathfrak{U}$
is a linear subspace of $\mathfrak{L}_A(H)$, $J:\mathfrak{U}\to\mathfrak{U}$,
and $\Gamma:\mathfrak{U}\to B(H)$. The collection $(\mathfrak{U}, J, \Gamma)$ is an
admissible triplet for the operator $A$ and $\mathfrak{U}$ is the space of
admissible perturbations, if the following properties hold.

1) The space $\mathfrak{U}$ is a Banach space (with respect to some norm
$\|\cdot\|_{\mathfrak{U}}$) that is continuously embedded in $\mathfrak{L}_A(H)$.

2) $J$ and $\Gamma$ are continuous linear operators and, moreover, $J$ is a projection.

3) $(\Gamma X)D(A)\subset D(A)$ and
\begin{equation}\label{adA}
A(\Gamma X)x-(\Gamma X)Ax=(X-JX)x
\end{equation}
for all $x\in D(A)$, $X\in\mathfrak{U}$. Moreover, $Y=\Gamma X\in B(H)$
is the unique solution of the equation
\[
AY-YA=X-JX,
\]
that satisfies $JY=0$.

4) $X\Gamma Y$, $(\Gamma X)Y\in\mathfrak{U}$ for all $X$, $Y\in\mathfrak{U}$,
and there is a constant $\gamma>0$ such that
\[
\|\Gamma\|\leqslant \gamma, \qquad \max\{\|X\Gamma Y\|_{\mathfrak{U}},
\|(\Gamma X)Y\|_{\mathfrak{U}}\}\leqslant \gamma\|X\|_{\mathfrak{U}}\|Y\|_{\mathfrak{U}}.
\]

5) For every $X\in\mathfrak{U}$ and $\delta>0$ there exists a number
$\lambda_\delta\in\rho(A)$ such that~$\|X(A-\lambda_\delta I)^{-1}\|<\delta$.
\end{definition}

In order to obtain an idea about this definition, one should think of the operators
involved in terms of their matrices. We suppose that the operator $A$ is represented
by a diagonal matrix and the operator $B$ has a matrix with some kind of off--diagonal
decay. The transform $J$ should be thought of as a projection that picks out the
main diagonal of an infinite matrix, whereas the transform $\Gamma$ annihilates the main
diagonal. Moreover, in this case, the operator $\Gamma$ enhances the off--diagonal damping.

Now we formulate the main theorem of the method of similar operators.
\begin{theorem}[\cite{BaskKrUsk}]\label{thsimgen}
Suppose that $(\mathfrak{U}, J, \Gamma)$ is an admissible triplet for the operator $A$
and the operator $B$ belongs to $\mathfrak{U}$. If
\begin{equation}\label{mainest}
\gamma\|J\|_{\mathfrak{U}}\|B\|_{\mathfrak{U}}<\frac{1}{4},
\end{equation}
where $\gamma$ is defined in Definition~\ref{defadm}, then the operator $A-B$ is similar
to the operator $A-JX_*$, where $X_*$ is a solution of the nonlinear equation
\[
X=B\Gamma X-(\Gamma X)JX+B=\Phi(X).
\]
The operator $\Phi: \mathfrak{U}\to\mathfrak{U}$ is a contraction and has unique
fixed point $X_*$ in the ball
$\{X\in\mathfrak{U}: \|X-B\|_{\mathfrak{U}}\leqslant 3\|B\|_{\mathfrak{U}}\}$,
which can be found as a limit of simple iterations: $X_0=0$, $X_1=\Phi(X_0)=B$, etc.
The invertible operator $I+\Gamma X_*\in B(H)$ is the similarity transform
of $A-B$ to $A-JX_*$.
\end{theorem}

\section{Similarity transformation for the operator $\mathcal{H}(\alpha)$}
In this section we apply the scheme from Section~\ref{sec4} to the operator
$\mathcal{H}(\alpha)$. The Hilbert space $H$ will be the space $L^2(0, 1)$.
We use the even--order operator $\mathcal{A}$ as the
unperturbed operator~$A$. The spectral properties of this operator are well known. It is
self--adjoint operator with a compact resolvent. Its eigenvalues, associated
eigenfunctions and the Riesz projections have the form
\begin{equation}\label{2.0}
\mu_n:=(\pi n)^{2k},\quad e_n(x):=\sqrt{2}\sin\pi n x,
\quad \mathcal{P}_nx:= (x, e_n)_{L^2(0, 1)}e_n, \quad x\in [0, 1], \quad n\geqslant 1,
\end{equation}
for the Dirichlet type boundary conditions and
\begin{equation}\label{3.0}
\mu_n:=(\pi n)^{2k},\quad e_n(x):=\sqrt{2}\cos\pi nx,
\quad \mathcal{P}_nx:= (x, e_n)_{L^2(0, 1)}e_n, \quad x\in [0, 1],\quad
n\geqslant 0,
\end{equation}
for the Neumann type boundary conditions. All eigenvalues are simple and
the system of eigenfunctions forms an orthonormal basis in $L^2(0, 1)$. Since the
spectral properties of the operator $\mathcal{A}$ with different boundary conditions
are quite similar, and the spectrums are structurally identical (except for one point),
we use the same notations, where this will not cause misunderstanding.

Now we construct an admissible triplet for the unperturbed operator $\mathcal{A}$.
As the space of admissible perturbations $\mathfrak{U}$ we take the space
$B(L^2(0, 1))$. Further, we consider a bounded operator
$\mathcal{X}: D(\mathcal{X})\subset L^2(0, 1)\to L^2(0, 1)$. This operator
is completely characterized by its operator matrix, which will be identified with
it and denoted by the same symbol. The entries of the matrices are the operators
$\mathcal{X}_{lj}=\mathcal{P}_l\mathcal{X}\mathcal{P}_j$, $l$, $j\in\mathbb{J}$,
where $\mathbb{J}=\{\mathbb{N}, \mathbb{N}\cup\{0\}\}$. Now we define the
operators $J$, $\Gamma: \mathfrak{U}\to\mathfrak{U}$ via their operator matrices
by the formulas
\begin{equation}\label{j}
(J\mathcal{X})_{lj}=\delta_{l-j}\mathcal{X}_{lj}, \quad l, j\in\mathbb{J}, \quad
\mathcal{X}\in\mathfrak{U},
\end{equation}
\begin{equation}\label{gamma}
(\Gamma\mathcal{X})_{lj}=\frac{\mathcal{X}_{lj}}{\mu_l-\mu_j},
\qquad l\ne j, \qquad (\Gamma\mathcal{X})_{ll}=0, \qquad l, j\in\mathbb{J},
\qquad \mathcal{X}\in\mathfrak{U}.
\end{equation}
It follows from \cite[Lemma~3.4]{BaskKrUsk} and \cite[Theorem~1.6]{Bask83} that these operators are well--defined,
bounded, the operator $J$ is a projection, and the following estimates hold
\begin{equation}\label{estJG}
\|J\|_{B(L^2(0, 1))}=1, \qquad \|\Gamma\|_{B(L^2(0, 1))}\leqslant\frac{\pi}{2\gamma_n}\leqslant\frac{2}{\gamma_n},
\qquad \gamma_n=\min_{l\ne n\in\mathbb{J}}|\mu_n-\mu_l|.
\end{equation}
Note that these operators satisfy the property $J((\Gamma\mathcal{X})J\mathcal{Y})=0$
for all $\mathcal{X}$, $\mathcal{Y}\in B(L^2(0, 1))$.

Thus, we have constructed a triplet $(\mathfrak{U}, J, \Gamma)$.
According to Definition~\ref{defadm}, we now establish that it is admissible triplet.
\begin{lemma}\label{lh1}
The triplet $(\mathfrak{U}, J, \Gamma)$ is admissible.
\end{lemma}
\begin{proof}
Property~1) of Definition~\ref{defadm} is immediate from the definitions of
$\mathfrak{U}=B(L^2(0, 1))$. Property 2) follows from the definitions of $J$ and $\Gamma$
and \cite[Lemma~3.4]{BaskKrUsk}.

Now we prove Property~3). The proof is identical for both the Dirichlet and Neumann type boundary conditions.
Let $x\in D(\mathcal{A})$ and $\lambda\in\rho(\mathcal{A})$. There exists
$y\in L^2(0, 1)$ such that
\[
x=(\mathcal{A}-\lambda I)^{-1}y=\sum_{j\in\mathbb{J}}\frac{1}{\mu_j-\lambda}\mathcal{P}_jy,
\]
where $\mathcal{P}_j$, $j\in\mathbb{J}$, are defined by \eqref{2.0} and \eqref{3.0}. Then
for $\mathcal{X}\in\mathfrak{U}$ we get
\begin{align*}
(\Gamma\mathcal{X})x &= (\Gamma\mathcal{X})(\mathcal{A}-\lambda I)^{-1}y=
\sum\limits_{\substack{l, j\in\mathbb{J} \\ l\ne j}}
\frac{\mathcal{P}_l\mathcal{X}\mathcal{P}_jy}{(\mu_l-\mu_j)(\mu_j-\lambda)} \\
&= \sum\limits_{\substack{l, j\in\mathbb{J} \\ l\ne j}}
\frac{\mathcal{P}_l\mathcal{X}\mathcal{P}_jy}{(\mu_l-\mu_j)(\mu_l-\lambda)}+
\sum\limits_{\substack{l, j\in\mathbb{J} \\ l\ne j}}
\frac{\mathcal{P}_l\mathcal{X}\mathcal{P}_jy}{(\mu_l-\lambda)(\mu_j-\lambda)} \\
&= (\mathcal{A}-\lambda I)^{-1}(\Gamma\mathcal{X})y+(\mathcal{A}-\lambda I)^{-1}
(\Gamma\mathcal{X})(\mathcal{A}-\lambda I)^{-1}y \\
&= (\mathcal{A}-\lambda I)^{-1}(\Gamma\mathcal{X})(x+y)\in D(\mathcal{A}).
\end{align*}
Thus, $(\Gamma\mathcal{X})D(\mathcal{A})\subset D(\mathcal{A})$. Moreover,
these identities show that the matrices of the bounded operators
$(\Gamma\mathcal{X})(\mathcal{A}-\lambda I)^{-1}
-(\mathcal{A}-\lambda I)^{-1}(\Gamma\mathcal{X})$ and $(\mathcal{A}-\lambda
I)^{-1}(\mathcal{X}-J\mathcal{X})(\mathcal{A}-\lambda I)^{-1}$ coincide.
This gives the formula \eqref{adA}. It follows from \eqref{j} and \eqref{gamma}
that the identity $J(\Gamma\mathcal{X})=0$ holds. Therefore, Property 3) is proved.

Now we verify Property~4). Assume that $\mathcal{X}$, $\mathcal{Y}\in B(L^2(0, 1))$.
Using the estimates \eqref{estJG}, we obtain
\[
\|\mathcal{X}\Gamma\mathcal{Y}\|_{B(L^2(0, 1))}\leqslant \|\mathcal{X}\|_{B(L^2(0, 1))}
\|\Gamma\|_{B(L^2(0, 1))}\|\mathcal{Y}\|_{B(L^2(0, 1))}\leqslant
2\gamma_n^{-1}\|\mathcal{X}\|_{B(L^2(0, 1))}\|\mathcal{Y}\|_{B(L^2(0, 1))}.
\]
Obviously, this inequality holds for $\|(\Gamma\mathcal{X})\mathcal{Y}\|_{B(L^2(0, 1))}$.
Property 4) is proved with constant $2\gamma_n^{-1}$.

Finally, Property~5) follows directly from the boundedness of the operator
$\mathcal{X}(\mathcal{A}-\lambda I)^{-1}$ (see also~\cite[Definition~2.1]{BaskKrUsk}).
Lemma is proved.
\end{proof}

For the perturbation $B$ we use the operator $\mathcal{B}(\alpha)$ of the form \eqref{B}.
Recall that the function $V$ satisfies the condition \eqref{Vc1lalpha}.
Obviously, $\mathcal{B}(\alpha)$ belongs to $\mathfrak{L}_{\mathcal{A}}(L^2(0, 1))$ and,
moreover, is bounded. Indeed, using the inequality \eqref{Vc1lalpha}, we get
\begin{align*}
\|\mathcal{B}(\alpha)y\|^2_{L^2(0, 1)} = \int_0^1|(\mathcal{B}(\alpha)y)(x)|^2\,dx &=
\int_0^1|V(x)y(\alpha x)|^2\,dx \\
&\leqslant\frac{\|V\|_{L^\infty(0, 1)}^2
\|y\|^2_{L^2(0, 1)}}{\alpha}\leqslant c_0^2\|y\|^2_{L^2(0, 1)}.
\end{align*}
These estimates immediately give
\begin{equation}\label{estB}
\|\mathcal{B}(\alpha)\|_{B(L^2(0, 1))}\leqslant c_0.
\end{equation}
Recall that the constant $c_0$ is independent of $\alpha$. Therefore, the operator
$\mathcal{B}(\alpha)$ belongs to the space of admissible
perturbation $\mathfrak{U}=B(L^2(0, 1))$ and we can formulate the main theorem of
similarity for the operator $\mathcal{H}(\alpha)=\mathcal{A}-\mathcal{B}(\alpha)$.
\begin{theorem}\label{thmainsimil}
If the constant $c_0<1/8$, then the operator $\mathcal{H}(\alpha)$ is similar
to the operator $\mathcal{A}-J\mathcal{X}_*$, where $\mathcal{X}_*$ is a solution
of the nonlinear equation
\begin{equation}\label{mainEq}
\mathcal{X}=\mathcal{B}(\alpha)\Gamma\mathcal{X}-(\Gamma\mathcal{X})J\mathcal{X}
+\mathcal{B}(\alpha).
\end{equation}
Moreover, the invertible operator $I+\Gamma\mathcal{X}_*$ is the similarity transform
of $\mathcal{H}(\alpha)$ to $\mathcal{A}-J\mathcal{X}_*$ and the following estimate holds
\begin{equation}\label{x-b}
\|\mathcal{X}_*\|_{B(L^2(0, 1))}\leqslant 4c_0,
\end{equation}
where $c_0>0$ is independent of $\alpha$.
\end{theorem}
\begin{proof}
It follows from Lemma~\ref{lh1} that the triplet $(B(L^2(0, 1)), J, \Gamma)$ is admissible.
We verify the estimate \eqref{mainest}. The formulas \eqref{2.0},
\eqref{3.0}, and \eqref{estJG} give $\gamma_n\geqslant 1$ for $n\in\mathbb{J}$.
Using the relations \eqref{estJG} and \eqref{estB}, we obtain
\[
\|J\|_{B(L^2(0, 1))}\|\mathcal{B}(\alpha)\|_{B(L^2(0, 1))}\|\Gamma\|_{B(L^2(0, 1))}
\leqslant \frac{2c_0}{\gamma_n}\leqslant 2c_0,
\]
where $c_0$ is independent of $\alpha$. Recall that the constant $c_0$ estimates
the potential $V$ (see the formula~\eqref{Vc1lalpha}). Therefore, choosing $c_0<1/8$,
we can always ensure that the inequality \eqref{mainest} holds.
It follows from Theorem~\ref{thsimgen} that the operator
$\mathcal{H}(\alpha)$ is similar to $\mathcal{A}-J\mathcal{X}_*$, the similarity
transform is the operator $I+\Gamma\mathcal{X}_*$, and the map
$\mathcal{B}(\alpha)\Gamma\mathcal{X}-(\Gamma\mathcal{X})J\mathcal{X}
+\mathcal{B}(\alpha): B(L^2(0, 1))\to B(L^2(0, 1))$ is a contraction and
has unique fixed point $\mathcal{X}_*$ in the ball
\[
\{\mathcal{X}\in B(L^2(0, 1)): \|\mathcal{X}_*-\mathcal{B}(\alpha)\|_{B(L^2(0, 1))}
\leqslant 3\|\mathcal{B}(\alpha)\|_{B(L^2(0, 1))}\}.
\]
Then the relations
\[
\|\mathcal{X}_*\|_{B(L^2(0, 1))}
\leqslant
\|\mathcal{X}_*-\mathcal{B}(\alpha)\|_{B(L^2(0, 1))}+\|\mathcal{B}(\alpha)\|_{B(L^2(0, 1))}
\leqslant 4\|\mathcal{B}(\alpha)\|_{B(L^2(0, 1))}
\]
and the inequality~\eqref{estB} give~\eqref{x-b}. Theorem is proved.
\end{proof}

\section{Proof of the main result}\label{sec6}
In this section we prove Theorem~\ref{th1}.
\begin{proof}[Proof of Theorem~\ref{th1}.]
It follows from Theorem~\ref{thmainsimil} that the operator $\mathcal{H}(\alpha)$
is similar to the operator $\mathcal{A}-J\mathcal{X}_*$. Therefore, using
Lemma~\ref{lemsim}, we obtain that the spectra of the operators $\mathcal{H}(\alpha)$
and $\mathcal{A}-J\mathcal{X}_*$ are coincided. Thus, one can perform further calculations
with the operator $\mathcal{A}-J\mathcal{X}_*$. We transform it as follows
\[
\mathcal{A}-J\mathcal{X}_*=\mathcal{A}-J\mathcal{B}(\alpha)
-J\big(\mathcal{B}(\alpha)\Gamma\mathcal{B}(\alpha)\big)
-J\big(\mathcal{X}_*-\mathcal{B}(\alpha)-\mathcal{B}(\alpha)\Gamma\mathcal{B}(\alpha)\big).
\]
By applying $e_n$ and taking the scalar product of both sides of the last
identity with $e_n$, we obtain
\begin{equation}\label{abstrasympt}
\begin{aligned}
\lambda_n = (\pi n)^{2k}-(J\mathcal{B}(\alpha)e_n, e_n)_{L^2(0, 1)}
&-\big(J(\mathcal{B}(\alpha)\Gamma
\mathcal{B}(\alpha))e_n, e_n\big)_{L^2(0, 1)} \\
&-\big(J(\mathcal{X}_*-\mathcal{B}(\alpha)
-\mathcal{B}(\alpha)\Gamma\mathcal{B}(\alpha))e_n, e_n\big)_{L^2(0, 1)}.
\end{aligned}
\end{equation}

We consider all the terms of the last identity independently. Now we find the estimates
of the error term. Using Equation \eqref{mainEq}, we obtain
\[
\mathcal{X}_*-\mathcal{B}(\alpha)-\mathcal{B}(\alpha)\Gamma\mathcal{B}(\alpha)
=\mathcal{B}(\alpha)\Gamma\big(\mathcal{X}_*-\mathcal{B}(\alpha)\big)
-(\Gamma\mathcal{X}_*)J\mathcal{X}_*.
\]
By applying the operator $J$ to both sides of the last equality and using the property
of the operators $J$ and $\Gamma$, we get
\[
J\big(\mathcal{X}_*-\mathcal{B}(\alpha)-\mathcal{B}(\alpha)\Gamma\mathcal{B}(\alpha)\big)=
J\Big(\mathcal{B}(\alpha)\Gamma\big(\mathcal{X}_*-\mathcal{B}(\alpha)\big)\Big).
\]
This identity, the boundedness of the operators $\mathcal{B}(\alpha)$, $\mathcal{X}_*$,
$J$, $\Gamma$, and the inequalities \eqref{estJG} and~\eqref{estB} give
\begin{equation}\label{prom1}
\begin{aligned}
&\Big|\big(J(\mathcal{X}_*-\mathcal{B}(\alpha)
-\mathcal{B}(\alpha)\Gamma\mathcal{B}(\alpha))e_n, e_n\big)_{L^2(0, 1)}\Big|\\
&\leqslant\|J(\mathcal{X}_*-\mathcal{B}(\alpha)
-\mathcal{B}(\alpha)\Gamma\mathcal{B}(\alpha))\|_{B(L^2(0, 1))}
=\|J\big(\mathcal{B}(\alpha)\Gamma\big(\mathcal{X}_*
-\mathcal{B}(\alpha)\big)\big)\|_{B(L^2(0, 1))} \\
&\leqslant \|J\|_{B(L^2(0, 1))}\|\mathcal{B}(\alpha)\|_{B(L^2(0, 1))}
\|\Gamma\|_{B(L^2(0, 1))}\|\mathcal{X}_*-\mathcal{B}(\alpha)\|_{B(L^2(0, 1))} \\
&\leqslant \frac{2c_0}{\gamma_n}\|\mathcal{X}_*-\mathcal{B}(\alpha)\|_{B(L^2(0, 1))}.
\end{aligned}
\end{equation}
It remains to estimate the norm. Using again Equation~\eqref{mainEq} and
the inequalities \eqref{estJG}, \eqref{estB},~\eqref{x-b}, we obtain
\begin{align*}
&\|\mathcal{X}_*-\mathcal{B}(\alpha)\|_{B(L^2(0, 1))} =
\|\mathcal{B}(\alpha)\Gamma\mathcal{X}_*
-(\Gamma\mathcal{X}_*)J\mathcal{X}_*\|_{B(L^2(0, 1))} \\
&\leqslant \|\mathcal{B}(\alpha)\|_{B(L^2(0, 1))}\|\Gamma\|_{B(L^2(0, 1))}
\|\mathcal{X}_*\|_{B(L^2(0, 1))}+\|\Gamma\|_{B(L^2(0, 1))}
\|\mathcal{X}_*\|_{B(L^2(0, 1))}^2\leqslant \frac{40c_0^2}{\gamma_n}.
\end{align*}
Substituting this estimate into \eqref{prom1}, we obtain
\begin{equation}\label{error}
\Big|\big(J(\mathcal{X}_*-\mathcal{B}(\alpha)
-\mathcal{B}(\alpha)\Gamma\mathcal{B}(\alpha))e_n, e_n\big)_{L^2(0, 1)}\Big|
\leqslant \frac{80c_0^3}{\gamma_n^2}.
\end{equation}
Finally, the third formula in~\eqref{estJG} implies that
\begin{equation}\label{gamma-1}
\frac{1}{\gamma_n}=\frac{1}{\min\limits_{l\ne n\in\mathbb{J}}|\mu_l-\mu_n|}=
\frac{1}{\pi^{2k}(l^{2k}-n^{2k})}\leqslant \frac{1}{\pi^{2k}(2n-1)n^{2k-2}}\leqslant
\frac{1}{\pi^{2k}n^{2k-1}}.
\end{equation}
Therefore, we get the following estimate for the error term:
\begin{equation}\label{error}
\Big|\big(J(\mathcal{X}_*-\mathcal{B}(\alpha)
-\mathcal{B}(\alpha)\Gamma\mathcal{B}(\alpha))e_n, e_n\big)_{L^2(0, 1)}\Big|
\leqslant \frac{80c_0^3}{\pi^{4k}n^{4k-2}}.
\end{equation}

Now we return to the representation~\eqref{abstrasympt} and compute the first
approximation term. Using the relations \eqref{B} and \eqref{2.0} for
the Dirichlet type boundary conditions, we obtain
\begin{equation}\label{firstDir}
(J\mathcal{B}(\alpha)e_n, e_n)_{L^2(0, 1)}=-2\int_0^1V(x)\sin\pi nx\sin\pi n\alpha x\,dx.
\end{equation}
Similarly, the identities \eqref{B} and \eqref{3.0} for the Neumann type boundary
conditions give
\begin{equation}\label{firstNeu}
(J\mathcal{B}(\alpha)e_n, e_n)_{L^2(0, 1)}=-2\int_0^1V(x)\cos\pi nx\cos\pi n\alpha x\,dx.
\end{equation}

It remains to compute the second approximation term in~\eqref{abstrasympt}. Using
the matrix representation \eqref{gamma} of the operator $\Gamma$, we obtain
\begin{equation}\label{second}
\big(J(\mathcal{B}(\alpha)\Gamma\mathcal{B}(\alpha))e_n, e_n\big)_{L^2(0, 1)}
=\sum_{j\in\mathbb{J}\setminus\{n\}}\frac{b_{nj}^Mb_{jn}^M}{\mu_j-\mu_n}
=\frac{1}{\pi^{2k}}\sum_{j\in\mathbb{J}\setminus\{n\}}\frac{b_{nj}^Mb_{jn}^M}{j^{2k}-n^{2k}},
\end{equation}
where $M=\{D, N\}$ and $b_{nj}^M$ has the form \eqref{bnjD} or \eqref{bnjN}.

Now we determine the order of decrease of this term. Using the boundedness of the operators
$\mathcal{B}(\alpha)$, $J$, $\Gamma$, and the estimates \eqref{estB}, \eqref{estJG},
\eqref{gamma-1}, we get
\begin{align*}
\Big|\big(J(\mathcal{B}(\alpha)\Gamma\mathcal{B}(\alpha))e_n, e_n\big)_{L^2(0, 1)}\Big|
&\leqslant \|J(\mathcal{B}(\alpha)\Gamma\mathcal{B}(\alpha))\|_{B(L^2(0, 1))} \\
&\leqslant \|J\|_{B(L^2(0, 1))}\|\mathcal{B}(\alpha)\|_{B(L^2(0, 1))}^2
\|\Gamma\|_{B(L^2(0, 1))}\leqslant \frac{2c_0^2}{\pi^{2k}n^{2k-1}}.
\end{align*}
Thus, we have established all needed relations for $\lambda_n$. Substituting
\eqref{error}, \eqref{firstDir}, \eqref{firstNeu}, and \eqref{second}
into \eqref{abstrasympt}, we obtain the asymptotics \eqref{asymptDir} and \eqref{asymptNeu}.
Theorem is proved.
\end{proof}

\section{Appendix}

As noted in Remark~\ref{rem2}, the requirement that the constant $c_0$ be small in
Theorem~\ref{th1} can be relaxed. Theorem~\ref{th1} can be reformulated without this
condition, but then the asymptotics \eqref{asymptDir} and \eqref{asymptNeu} hold
for sufficiently large indices. Now we provide the necessary proof.

The smallness of $c_0$ played an important role in the proof of Theorem~\ref{thmainsimil}.
We reformulate this theorem for an arbitrary constant $c_0$. Together with
the operators $J$ and $\Gamma$ of the form \eqref{j} and \eqref{gamma} we consider
the sequences of the operators
\[
J_l\mathcal{X}=J(\mathcal{X}-\mathcal{P}_{(l)}\mathcal{X}\mathcal{P}_{(l)})
+\mathcal{P}_{(l)}\mathcal{X}\mathcal{P}_{(l)},
\qquad \Gamma_l\mathcal{X}=\Gamma(\mathcal{X}-\mathcal{P}_{(l)}\mathcal{X}\mathcal{P}_{(l)}),
\]
where $l\in\mathbb{J}$, $\mathcal{P}_{(l)}$ is defined as $\sum_{j=1}^l\mathcal{P}_j$
under Dirichlet type boundary conditions and as $\sum_{j=0}^l\mathcal{P}_j$
under Neumann type boundary conditions. Recall that
the projections $\mathcal{P}_l$, $l\in\mathbb{J}$, have the form \eqref{2.0} or \eqref{3.0}.
If we consider these operators as matrices, then $J_l$ removes part of the main diagonal,
and $\Gamma_l$ removes a large block of size $l\times l$. Thus, we can control the
size of the removed parts. Moreover, $J_l$ and $\Gamma_l$ differ from the
corresponding operators $J$ and $\Gamma$ by a finite rank operator. Since
$\mathcal{B}(\alpha)\in B(L^2(0, 1))$, the operators $J_l\mathcal{B}(\alpha)$
and $\Gamma_l\mathcal{B}(\alpha)$ are well--defined and the estimates \eqref{estJG}
hold for these operators.

According to the scheme of the method of similar operators, we must establish
that the new triplet $(B(L^2(0, 1)), J_l, \Gamma_l)$ is admissible triplet.
Since $J_l$ and $\Gamma_l$ differ from the corresponding operators $J$ and $\Gamma$
by a finite rank operator, the proof of Lemma~\ref{lh1} remains unchanged.

Now we reformulate Theorem~\ref{thmainsimil}.
\begin{theorem}\label{thmainsimil1}
There exists a sufficiently large $n\geqslant l+1$ independent of $\alpha$
such that the operator $\mathcal{H}(\alpha)$ is similar to the operator
$\mathcal{A}-J_n\mathcal{X}_*$, where $\mathcal{X}_*$ is a solution
of the nonlinear equation \eqref{mainEq} with $J=J_n$ and $\Gamma=\Gamma_n$. Moreover, the invertible operator
$I+\Gamma_n\mathcal{X}_*$ is the similarity transform of $\mathcal{H}(\alpha)$ to
$\mathcal{A}-J_n\mathcal{X}_*$ and the estimate \eqref{x-b} holds.
\end{theorem}
\begin{proof}
We need to verify the estimate \eqref{mainest}. Using the formulas
\eqref{estJG} and \eqref{estB}, we obtain
\begin{equation}\label{prom2}
\|J_n\|_{B(L^2(0, 1))}\|\mathcal{B}(\alpha)\|_{B(L^2(0, 1))}\|\Gamma_n\|_{B(L^2(0, 1))}
\leqslant \frac{2c_0}{\gamma_n},
\end{equation}
where $c_0$ is independent of $\alpha$. Recall that $\gamma_n$ is defined by the eigenvalues
of the operator $\mathcal{A}$. Since $\gamma_n^{-1}\to 0$,
as $n\to\infty$, we can always choose a large enough $n$ for which the inequality \eqref{prom2}
is less than $1/4$. It follows from Theorem~\ref{thsimgen} that the operator
$\mathcal{H}(\alpha)$ is similar to $\mathcal{A}-J_n\mathcal{X}_*$ and the similarity
transform is the operator $I+\Gamma_n\mathcal{X}_*$. The proof of the rest of
this theorem completely repeats the proof of Theorem~\ref{thmainsimil}.
\end{proof}

Finally, it remains to make the necessary additions to the proof of Theorem~\ref{th1}.
It follows from Theorem~\ref{thmainsimil1} that the operator $\mathcal{H}(\alpha)$
is similar to the operator $\mathcal{A}-J_n\mathcal{X}_*$. Therefore, using
Lemma~\ref{lemsim}, we obtain that the spectra of the operators $\mathcal{H}(\alpha)$
and $\mathcal{A}-J_n\mathcal{X}_*$ coincide. Moreover, there exists $n\in\mathbb{J}$
such that $\mathcal{A}-J_n\mathcal{X}_*$ commutes with all projections
$\mathcal{P}_{(l)}$, $\mathcal{P}_n$, $n\geqslant l+1$.
Therefore, the subspaces $\mathfrak{H}_{(l)}=\mathrm{Im}\,\mathcal{P}_{(l)}$,
$\mathfrak{H}_n=\mathrm{Im}\,\mathcal{P}_n$, $n\geqslant l+1$, are invariant for the operator
$\mathcal{A}-J_n\mathcal{X}_*$. It follows from Theorem~\ref{th0} that the operator
$\mathcal{H}(\alpha)$ (as well as $\mathcal{A}-J_n\mathcal{X}_*$) has a compact resolvent.
Thus, if $\lambda_0\in\sigma(\mathcal{A}-J_n\mathcal{X}_*)$, then there exists
an eigenvector $x_0\in \mathfrak{D}(\mathcal{A})$ such that
$(\mathcal{A}-J_n\mathcal{X}_*)x_0=\lambda_0x_0$. Therefore, it follows from
\eqref{j} that
\begin{equation}\label{polLp}
\mathfrak{A}_{(l)}\mathcal{P}_{(l)}x_0=\lambda_0\mathcal{P}_{(l)}x_0, \qquad
\qquad \mathfrak{A}_n\mathcal{P}_nx_0=\lambda_0\mathcal{P}_nx_0, \qquad \qquad
n\geqslant l+1,
\end{equation}
where $\mathfrak{A}_{(l)}=(\mathcal{A}-J_n\mathcal{X}_*)| \mathfrak{H}_{(l)}$ and
$\mathfrak{A}_n=(\mathcal{A}-J_n\mathcal{X}_*)|\mathfrak{H}_n$ are the restrictions
of the operator $\mathcal{A}-J_n\mathcal{X}_*$ to the subspaces
$\mathfrak{H}_{(l)}$ and $\mathfrak{H}_n$, $n\geqslant l+1$, respectively.
Since the system of projections $\mathcal{P}_{(l)}$, $\mathcal{P}_n$, $n\geqslant l+1$,
is a resolution of the identity, the formula \eqref{polLp} implies that at least one
of the vectors $\mathcal{P}_nx_0$, $n\geqslant l+1$, $\mathcal{P}_{(l)}x_0$,
is nonzero. Thus, $\lambda_0$ is an eigenvalue of the corresponding operator
in the family $\mathfrak{A}_{(l)}$, $\mathfrak{A}_n$, $n\geqslant l+1$. Therefore,
we have the embedding
\[
\sigma(\mathcal{H}(\alpha))=\sigma(\mathcal{A}-J_n\mathcal{X}_*)\subset
\sigma(\mathfrak{A}_{(l)})\bigcup \bigg(\bigcup\limits_{n\geqslant l+1}
\sigma(\mathfrak{A}_n)\bigg).
\]
The reverse embedding obviously holds. Hence
\begin{equation}\label{polsigmaL}
\sigma(\mathcal{H}(\alpha))=\sigma(\mathcal{A}-J_n\mathcal{X}_*)
= \sigma(\mathfrak{A}_{(l)})\bigcup \Big(\bigcup\limits_{n\geqslant l+1}
\sigma(\mathfrak{A}_n)\Big).
\end{equation}
Since the subspaces $\mathfrak{H}_{(l)}$ are finite--dimensional, it follows from
the representation \eqref{polsigmaL} that the set $\sigma(\mathfrak{A}_{(l)})$
contains at most $l$ elements. At the same time the subspaces $\mathfrak{H}_n$ are
one--dimensional and the set $\sigma(\mathfrak{A}_n)$ is a singleton and
$\sigma(\mathfrak{A}_n)=\{\lambda_n\}$.

The relations \eqref{polsigmaL} show that to obtain the eigenvalue asymptotics
of the operator $\mathcal{H}(\alpha)$ for large $n$, it suffices to describe the set
$\sigma(\mathfrak{A}_n)$. Consider again the operator $\mathcal{A}-J_n\mathcal{X}_*$.
We have
\[
\mathcal{A}-J_n\mathcal{X}_*=\mathcal{A}-J_n\mathcal{B}(\alpha)
-J_n\big(\mathcal{B}(\alpha)\Gamma_n\mathcal{B}(\alpha)\big)
-J_n\big(\mathcal{X}_*-\mathcal{B}(\alpha)-\mathcal{B}(\alpha)\Gamma_n\mathcal{B}(\alpha)\big).
\]
We consider the restriction of this operator to the subspace $\mathfrak{H}_n$. Then
\[
\mathfrak{A}_n=\mathcal{A}_n-(J\mathcal{B}(\alpha))_n-\big(J(\mathcal{B}(\alpha)\Gamma
\mathcal{B}(\alpha))\big)_n-\big(J(\mathcal{X}_*-\mathcal{B}(\alpha)
-\mathcal{B}(\alpha)\Gamma\mathcal{B}(\alpha))\big)_n, \quad n\geqslant l+1.
\]
In order to calculate all terms, we employ the matrix representation of these
operators in the basis of $\mathfrak{H}_n$. Next, we apply $e_n$ to both sides of
the last equality and take the inner product with $e_n$. After this, we completely
repeat the calculations in the proof of Theorem~\ref{th1} (see Section~\ref{sec6}). Therefore, the formulation of
Theorem~\ref{th1} remains completely the same for arbitrary $c_0$, but for
sufficiently large indices $n$.


\begin{thebibliography}{}
\bibitem{Ambartsumyan}
V.A.~Ambartsumyan. On the theory of brightness fluctuations in Milky Way.
Dokl. Akad. Nauk SSSR. 1944. V.~44. P.~244--247.

\bibitem{Ockendon}
J.R. Ockendon, A.B. Tayler. The dynamics of a current collection system for
an electric locomotive. Proc. R. Soc. London Ser. A Math. Phys. Eng. Sci. 1971.
V.~322. P.~447--468.

\bibitem{Hall}
A.J. Hall, G.C. Wake. A functional differential equation arising in the
modelling of cell growth. J. Aust. Math. Soc. Ser. B. 1989. V.~30. P.~424--435.

\bibitem{Gaver}
Jr.D.P. Gaver. An absorption probability problem. J. Math. Anal. Appl. 1964. V.~9(3). P.~384--393.

\bibitem{Ros2011}
L.E. Rossovskii. On the specreal stability of functional--differential equations.
Math. Notes. 2011. V.~90. P.~867--881.

\bibitem{Ros2017}
L.E. Rossovskii. Elliptic functional differential equations with contractions and
extensions of independent variables of the unknown function. J. Math. Sci. 2017.
V.~223. P.~351--493.

\bibitem{RosTovs2019}
L.E. Rossovskii, A.A. Tovsultanov. Elliptic functional differential equation with
affine transformation. J. Math. Anal. Appl. 2019. V.~480. 123403.

\bibitem{RosTovs2022}
L.E. Rossovskii, A.A. Tovsultanov. Functional--differential equations with dilation
and symmetry. Sib. Math. J. 2022. V.~63. P.~758--768.

\bibitem{PodSkub}
V.V. Pod'yapol'skii, A.L. Skubachevskii. On the completeness and basis property of
a system of root functions of strongly elliptic functional differential operators.
Russ. Math. Surv. 1996. V.~51(1). P.~169--170.

\bibitem{Skub}
A.L. Skubachevskii. Boundary--value problems for elliptic functional--differential
equations and their applications. Russ. Math. Surv. 2016. V.~71(5). P.~801--906.

\bibitem{Nedic}
D. Nedic, E. Catrnja. Spectral properties of some differential operators of
Sturm--Liouville type with homogeneous delay. Hacettepe J. Math. Stat. 2022. V.~51(3). P.~658--665.

\bibitem{Pikula2014}
M. Pikula, V. Vladicic, D. Nedic. Inverse Sturm--Liouville problems with homogeneous
delay. Sib. Math. J. 2014. V.~55(2). P.~301--308.

\bibitem{Pikula2013}
M. Pikula, V. Vladicic, O. Markovic. A solution to the inverse problem for the
Sturm--Liouville--type equation with a delay. Filomat. 2013. V.~27(7). P.~1237--1245.

\bibitem{Djuric}
N. Djuric. A new approach to inverse problems for Sturm--Liouville operators with
homogeneous delay. Bol. Soc. Mat. Mex. 2025. V.~31. 100.

\bibitem{BaskKrUsk}
A.G. Baskakov, I.A. Krishtal, N.B. Uskova. Similarity techniques in the spectral
analysis of perturbed operator matrices. J. Math. Anal. Appl. 2019. V.~477. P.~930--960.

\bibitem{BorPolIzv}
D.I. Borisov, D.M. Polyakov. Spectral asymptotics for Schr\"odinger operator
perturbed by translation operator. Izv. Math. 2025. V.~89(3). P.~442--460.

\bibitem{PolyakovM2AS}
D.M. Polyakov. Spectral analysis of an even order differential operator with square
integrable potential. Math. Meth. Appl. Sci. 2023. V.~46(5). P.~5483--5504.

\bibitem{BondButVas}
N.P. Bondarenko, S.A. Buterin, S.V. Vasiliev. An inverse spectral problem for
Sturm--Liouville operators with frozen argument. J. Math. Anal. Appl. 2019.
V.~472. P.~1028--1041.

\bibitem{BondHuYang}
Y.-T. Hu, N.P. Bondarenko, C.-F. Yang. Traces and inverse nodal problem for
Sturm--Liouville operators with frozen argument. Appl. Math. Lett. 2020. V.~102. 106096.

\bibitem{Akhmerova}
E.F. Akhmerova. Asymptotics of the spectrum of nonsmooth perturbations of
differential operators of order $2m$. Math. Notes. 2011. V.~90(6). P.~813--823.

\bibitem{BorPolM2AS}
D.I. Borisov, D.M. Polyakov. Spectral properties of Schr\"odinger operator with
translations and Neumann boundary conditions. Math. Meth. Appl. Sci. 2026.
V. 49(7). P.~7108--7127.

\bibitem{Kato}
T. Kato. Perturbation theory for linear operators. Grundlehren Math. Wiss.,
vol. 132, Springer--Verlag New York, Inc., New York, 1966.

\bibitem{Bask83}
A.G. Baskakov. Methods of abstract harmonic analysis in the theory of perturbations of linear
operators. Siberian Math. J. 1983. V.~24(1). P.~17--32.
\end{thebibliography}
\end{document}